\documentclass[12pt]{amsart}
\usepackage{amscd,amssymb}
\usepackage{amsthm,amsmath,amssymb}
\usepackage[matrix,arrow,curve]{xy}
\usepackage{color}
\newtheorem{theorem}[equation]{Theorem}
\newtheorem{lemma}[equation]{Lemma}
\newtheorem{corollary}[equation]{Corollary}

\theoremstyle{definition}
\newtheorem{example}[equation]{Example}
\newtheorem{definition}[equation]{Definition}

\theoremstyle{remark}
\newtheorem{remark}[equation]{Remark}

\makeatletter\@addtoreset{equation}{section}
\makeatother

\title{Alpha-invariants and purely log terminal blow-ups}

\author{Ivan Cheltsov, Jihun Park, and Constantin Shramov}

\address{ \emph{Ivan Cheltsov}\newline \textnormal{School of Mathematics, The
University of Edinburgh
\newline \medskip  Edinburgh EH9 3JZ, UK
\newline
National Research University Higher School of Economics, Laboratory of Algebraic Geometry,
\newline
6 Usacheva street, Moscow, 117312, Russia
\newline
\texttt{I.Cheltsov@ed.ac.uk}}}

\address{ \emph{Jihun Park}\newline \textnormal{Center for Geometry and Physics, Institute for Basic Science
\newline \medskip 77 Cheongam-ro, Nam-gu, Pohang, Gyeongbuk, 37673, Korea \newline
Department of
Mathematics, POSTECH
\newline
77 Cheongam-ro, Nam-gu, Pohang, Gyeongbuk,  37673, Korea \newline
\texttt{wlog@postech.ac.kr}}}

\address{\emph{Constantin Shramov}\newline \textnormal{Steklov Mathematical Institute of Russian Academy of Sciences
\newline
\medskip 8 Gubkina street, Moscow, 119991, Russia
\newline
National Research University Higher School of Economics, Laboratory of Algebraic Geometry,
\newline
6 Usacheva street, Moscow, 117312, Russia
\newline \texttt{costya.shramov@gmail.com }}}

\thanks{We assume that all varieties are defined over the field $\mathbb{C}$.}

\begin{document}

\begin{abstract}
We prove that the sum of the $\alpha$-invariants of two different Koll\'ar components of a Kawamata log terminal singularity is less than~$1$.
\end{abstract}

\maketitle

Let $V$ be a normal irreducible projective variety of dimension $n\geqslant 1$, and let $\Delta_V$ be an effective $\mathbb{Q}$-divisor on $V$.
Write
$$
\Delta_V=\sum_{i=1}^{r}a_i\Delta_i,
$$
where each $\Delta_i$ is a prime divisor, and each $a_i$ is a positive rational number.
Suppose that the log pair $(V,\Delta_{V})$ has at most Kawamata log terminal singularities.
Then, in particular, {each $a_i$ is less than $1$}.
Suppose also that the divisor $-(K_{V}+\Delta_V)$ is ample, so that~\mbox{$(V,\Delta_{V})$} is a log Fano variety.
Finally, suppose that $V$ is faithfully acted on by a finite group $G$ such that the divisor $\Delta_{V}$ is $G$-invariant.
Let $\alpha_G(V,\Delta_V)$ be the real number defined by
$$
\mathrm{sup}\left\{\lambda\in\mathbb{Q} \left|\ %
\aligned &\text{the pair}\ \big(V,\Delta_{V}+\lambda D_{V}\big)\ \text{has Kawamata log terminal singularities}\\
&\text{for every~$G$-invariant and effective $\mathbb{Q}$-divisor}\  D_{V}\sim_{\mathbb{Q}}-\Big(K_{V}+\Delta_{V}\Big)\\
\endaligned\right.\right\}.%
$$
This number is known as the $\alpha$-invariant of the log Fano variety $(V,\Delta_V)$,
or its global log canonical threshold (see \cite[Definition~3.1]{ChSh09}).
If $G$ is trivial, we put $\alpha(V,\Delta_{V})=\alpha_{G}(V,\Delta_{V})$.
\begin{example}\label{example:P1-log}
 The divisor $-(K_{\mathbb{P}^1}+\Delta_{\mathbb{P}^1})$ is ample if and only if $\sum_{i=1}^{r}a_i<2$.
One has
$$
\alpha(\mathbb{P}^1,\Delta_{\mathbb{P}^1})=\frac{1-\mathrm{max}(a_1,\ldots,a_r)}{2-\sum_{i=1}^{r}a_i}.
$$
\end{example}

We put $\alpha_G(V)=\alpha_{G}(V,\Delta_{V})$ if $\Delta_{V}=0$.

\begin{example}\label{example:P1-groups}
A finite group
$G$ acting faithfully on $\mathbb{P}^1$ is one of the following finite groups: the alternating group $\mathfrak{A}_5$, the symmetric group $\mathfrak{S}_4$, the alternating group
$\mathfrak{A}_4$, the dihedral group $\mathfrak{D}_{2m}$ of order $2m$, or the cyclic group $\boldsymbol{\mu}_{m}$ of order $m$ (including the case $m=1$, that is, the trivial group). The number $\frac{\alpha_G(\mathbb{P}^1)}{2}$ is equal to the length of the smallest $G$-orbit in~$\mathbb{P}^1$,
which gives
$$
\alpha_G\big(\mathbb{P}^1\big)=\left\{%
\aligned
&6 \mbox{ if $G\cong\mathfrak{A}_5$,}\\%
&3 \mbox{ if $G\cong\mathfrak{S}_4$,}\\%
&2 \mbox{ if $G\cong\mathfrak{A}_{4}$,}\\%
&1 \mbox{ if $G\cong\mathfrak{D}_{2m}$,}\\%
&\frac{1}{2} \mbox{ if $G\cong\boldsymbol{\mu}_{m}$.}\\%
\endaligned\right.%
$$
\end{example}

If both $\Delta_{V}=0$ and $G$ is trivial, we put $\alpha(V)=\alpha_{G}(V,\Delta_{V})$.
This is the most classical case.
Namely, if $V$ is a smooth Fano variety, then by \cite[Theorem~A.3]{ChSh08c} the number $\alpha(V)$ coincides
with the $\alpha$-invariant of $V$ defined by Tian in \cite{Tian}.
Its values were found or estimated in many cases.
For example, in the toric case the explicit formula for $\alpha(V)$ is given by Cheltsov and Shramov in \cite[Lemma~5.1]{ChSh08c}.
It gives $\alpha(\mathbb{P}^n)=\frac{1}{n+1}$, which can also be verified by an easy explicit computation.

The $\alpha$-invariants
of all del Pezzo surfaces with at worst Du Val singularities were computed
in~\mbox{\cite{Ch07b,Ch07c,Shi,Park-Won,Park-Won-PEMS,CheltsovKosta}}.
Furthermore, the $\alpha$-invariants of many non-Gorenstein singular del Pezzo surfaces that are quasi-smooth well-formed complete intersections in weighted projective spaces
were computed in~\mbox{\cite{ChPaSh10,ChSh2013zoo,KimPark}}.
The $\alpha$-invariants of many smooth and singular Fano threefolds were computed or estimated in \cite{JohnsonKollar,ChSh08c,Ch07a,Ch09a,Ch09b,KimOkadaWon}.
The $\alpha$-invariants of smooth Fano hypersurfaces were estimated in \cite{Ch01b,ChPa02,Pukhlikov,CheltsovParkWon}.

The $\alpha$-invariant plays an important role in K\"ahler geometry.
If $V$ is a smooth Fano variety, then $V$ admits a $G$-invariant K\"ahler--Einstein metric provided that
$$
\alpha_G\big(V\big)>\frac{\dim(V)}{\dim(V)+1}.
$$
This was proved by Tian in \cite{Tian}.
In \cite{Fujita2016}, this result was improved by Fujita.
He proved that $V$ admits a K\"ahler--Einstein metric if it is smooth and $\alpha(V)\geqslant\frac{\dim(V)}{\dim(V)+1}$.
In particular, all smooth hypersurfaces in $\mathbb{P}^{d}$ of degree $d$ are K\"ahler--Einstein,
because their $\alpha$-invariants are at least $\frac{d-1}{d}$ by \cite{Ch01b,ChPa02}.

The K-stability
of the log Fano variety $(V,\Delta_V)$ crucially depends on $\alpha(V,\Delta_V)$.
For instance, if
$$
\alpha\big(V,\Delta_V\big)<\frac{1}{\dim(V)+1},
$$
then the log Fano variety $(V,\Delta_V)$ is K-unstable by \cite[Theorem~3.5]{FujitaOdaka2016} and \cite[Lemma~5.5]{Fujita2017}.
This bound is sharp, since $\mathbb{P}^n$ is K-semistable and $\alpha(\mathbb{P}^n)=\frac{1}{n+1}$.
Vice versa, if~\mbox{$\alpha(V,\Delta_V)\geqslant\frac{\dim(V)}{\dim(V)+1}$},
then the log Fano variety $(V,\Delta_V)$ is K-semistable by \cite[Theorem~1.4]{OdakaSano} and \cite[Proposition~2.1]{Fujita2017plt}.

The $\alpha$-invariant also plays an important role in birational geometry.
It was first observed by Park in \cite{Park2001}, {where he proved a theorem that evolved into the following:}

\begin{theorem}[{{\cite[Theorem~1.5]{Ch07c}}}]
\label{theorem:Park}
Let $X$ be a variety with at most  terminal $\mathbb{Q}$-factorial singularities.
Suppose that there is a {proper} morphism $\phi\colon X\to Z$ such that $Z$ is a {smooth} curve, and $-K_X$ is $\phi$-ample.
Let $P$ be a point in $Z$, and let $E_X$ be the scheme fiber of $\phi$ over $P$.
Suppose that $E_X$ is irreducible, reduced, normal, and has at most Kawamata log terminal singularities,
so that $E_X$ is a Fano variety by the adjunction formula.
Suppose also that there is a commutative diagram
$$
\xymatrix{
X\ar@{->}[drr]_{\phi}\ar@{-->}[rrrr]^{\rho}&&&&\ Y\ar@{->}[dll]^{\psi}\\%
&&Z&&}
$$
such that $Y$ is a variety with at most  terminal $\mathbb{Q}$-factorial singularities,
$\psi$ is a {proper} morphism, the divisor $-K_{Y}$ is $\psi$-ample,
and $\rho$ is a birational map that induces an isomorphism
$$
X\setminus\mathrm{Supp}\big(E_X\big)\cong Y\setminus\mathrm{Supp}\big(E_Y\big),
$$
where $E_Y$ is the scheme fiber of $\psi$ over $P$.
Suppose, in addition, that $E_Y$ is irreducible.
Then $\rho$ is an isomorphism provided that~\mbox{$\alpha(E_X)\geqslant 1$}.
Moreover, if $E_Y$ is reduced, normal and has at most Kawamata log terminal singularities,
then $\rho$ is an isomorphism provided that $\alpha(E_X)+\alpha(E_Y)>1$.
\end{theorem}

Theorem~\ref{theorem:Park} gives a necessary condition in terms of $\alpha$-invariants
for the existence of a non-biregular fiberwise birational transformation of a Mori fibre space over a curve.
It follows from \cite[Theorem~1.1]{Krylov} that this condition is not a sufficient condition.
Nevertheless, the bound is  sharp (one can find many examples in \cite{Park2001,Park2002}).

\begin{example}
\label{example:P1}
Let $S$ be a $\mathbb{P}^1$-bundle over a curve. Then we have an elementary transformation to another $\mathbb{P}^1$-bundle over the same curve. Note that the $\alpha(\mathbb{P}^1)=\frac{1}{2}$ by Example~\ref{example:P1-groups}.
\end{example}

\begin{example}[{\cite[Example~5.8]{Co96}}]
\label{example:dP3-D4}
Let $S$ be a smooth cubic surface in $\mathbb{P}^3$
with an Eckardt point $O$. Denote by $L_{1}$, $L_{2}$, $L_{3}$ the lines in $S$ passing through $O$.
Put $X=S\times\mathbb{A}^{1}$, and let $\phi$ be the natural projection~\mbox{$X\to\mathbb{A}^{1}$}.
Let us identify $S$ with a fiber of $\phi$.
Then there is a commutative~diagram
$$
\xymatrix{
&U\ar@{->}[dl]_{\alpha}\ar@{-->}[rr]^{\psi}&&\overline{U}\ar@{->}[dr]^{\beta}&\\%
X\ar@{->}[drr]_{\phi}\ar@{-->}[rrrr]^{\rho}&&&&\ Y\ar@{->}[dll]^{\psi}\\%
&&\mathbb{A}^{1}&&}
$$ %
where $\alpha$ is the blow up of the point $O$, the map $\psi$ is the anti-flip along the
proper transforms of the curves $L_{1}$, $L_{2}$, $L_{3}$, and~$\beta$ is the contraction of the~proper transform of the surface $S$.
The scheme fiber of $\psi$ over the point $\phi(S)$ is a cubic surface in $\mathbb{P}^3$
that has one singular point of type $\mathbb{D}_{4}$.
Its $\alpha$-invariant is $\frac{1}{3}$ (see \cite[Theorem~1.4]{Ch07c}).
On the other hand, we have
$\alpha(S)=\frac{2}{3}$ (see \cite[Theorem~1.7]{Ch07b}).
\end{example}

\begin{example}[{\cite[Example~5.3]{Park2001}}]
\label{example:Jihun}
Let $X$ and $Y$ be subvarieties in $\mathbb{A}^{1}\times\mathbb{P}^{3}$ given by eqautions
$$
x^{3}+y^{2}z+z^{2}w+t^{12}w^{3}=0\quad \text{and}\quad x^{3}+y^{2}z+z^{2}w+w^{3}=0,
$$
respectively, where $t$ is a coordinate on $\mathbb{A}^{1}$, and
$(x:y:z:w)$ are homogeneous coordinates on $\mathbb{P}^{3}$. Then the projections
$\phi\colon X\to\mathbb{A}^{1}$ and
$\psi\colon Y\to\mathbb{A}^{1}$ are fibrations into cubic surfaces,
and the map
$$
(t,x,y,z,w)\mapsto(t,t^{2}x,t^{3}y,z,t^{6}w)
$$
gives a non-biregular birational fiberwise map $\rho\colon X\dasharrow Y$.
The fiber of $\phi$ over the point $t=0$ is a cubic surface that has one Du Val singular point of type $\mathbb{E}_{6}$,
so that its $\alpha$-invariant is $\frac{1}{6}$ (see \cite[Theorem~1.4]{Ch07c}),
and the  fiber of $\psi$ over the point $t=0$ is a smooth cubic surface with an Eckardt point,
so that its $\alpha$-invariant is $\frac{2}{3}$ (see \cite[Theorem~1.7]{Ch07b}).
\end{example}

The $\alpha$-invariant also plays an important role in singularity theory.
Let $U\ni P$ be a germ of a Kawamata log terminal singularity.
Then it follows from \cite[Lemma~1]{Xu2014} {(cf. \cite[Proposition~2.12]{Prokhorov})}
that there is a birational morphism $\phi\colon X\to U$ such that
its exceptional locus consists of a single prime divisor $E_X$ such that $\phi(E_X)=P$,
the log pair $(X,E_X)$ has purely log terminal singularities,
and the divisor $-(K_X+E_X)$ is $\phi$-ample.
Then
$$
-\big(K_X+E_X\big)\sim_{\mathbb{Q}} -\delta_X E_X
$$
for some positive rational number $\delta_X$.
Recall from \cite[Definition~2.1]{Prokhorov} that the birational morphism $\phi\colon X\to U$ is a \emph{purely log terminal blow-up} of the singularity $U\ni P$.

By \cite[Theorem~7.5]{Ko97}, the divisor $E_X$ is a normal variety that has rational singularities.
Moreover, it can be naturally equipped with a structure of a log Fano variety.
Let $R_{1},\ldots,R_{s}$ be all the irreducible components of the locus $\mathrm{Sing}(X)$ of codimension $2$ that are contained in $E_X$.
Put
$$
\mathrm{Diff}_{E_X}\big(0\big)=\sum_{i=1}^{s}\frac{m_{i}-1}{m_{i}}R_{i},
$$
where $m_{i}$ is the~smallest positive integer such that the divisor $m_{i}E_X$ is~Cartier in a general point of $R_{i}$.
Then $\mathrm{Diff}_{E_X}(0)$ is usually called the \emph{different} of the pair $(X,E_X)$.
One has
$$
-\delta_XE_X\Big\vert_{E_X}\sim_{\mathbb{Q}}-\Big(K_X+E_X\Big)\Big\vert_{E_X}\sim_{\mathbb{Q}} -\left(K_{E_X}+\mathrm{Diff}_{E_X}(0)\right).
$$
Furthermore, the singularities of the log pair $(E_X,\mathrm{Diff}_{E_X}(0))$ are Kawamata log terminal by Adjunction (
see \cite[3.2]{Sho93} or \cite[17.6]{Asterisque}).
This means that \mbox{$(E_X,\mathrm{Diff}_{E_X}(0))$} is a log Fano variety with Kawamata log terminal singularities,
because $-E_X$ is $\phi$-ample.

\begin{definition}[{cf. \cite[Definition~1.1]{LiXu}}]
\label{definition:Kollar-component}
The log Fano variety $(E_X,\mathrm{Diff}_{E_X}(0))$ is a Koll\'ar component of~\mbox{$U\ni P$}.
\end{definition}

Let us show how to compute $\alpha(E_X,\mathrm{Diff}_{E_X}(0))$ in three  simple cases.

\begin{example}[{{cf. \cite[Example~2.4]{Prokhorov}}}]
\label{example:Du-Val-different}
Let $U\ni P$ be a germ of a Du Val singularity, and $f\colon W\to U$ be the minimal resolution of this singularity.
Then the exceptional curves of $f$ are smooth rational curves whose self-intersections are $-2$,
and their dual graph is of type $\mathbb{A}_m$, $\mathbb{D}_m$, $\mathbb{E}_6$, $\mathbb{E}_7$, or $\mathbb{E}_8$.
Let $E_W$ be one of the exceptional curves that is chosen as follows.
If $U\ni P$ is not a singularity of type $\mathbb{A}_m$, let $E_W$ be the only $f$-exceptional curve
that intersects three other $f$-exceptional curves, i.e., $E_W$ is the ``fork'' of the dual graph.
If $U\ni P$ is a singularity of type $\mathbb{A}_m$, choose $E_W$ to be the $k$-th curve in the dual graph.
In this case, we may assume that $k\leqslant\frac{m+1}{2}$.
In all the cases, there exists a commutative diagram
$$
\xymatrix{
W\ar@{->}[rrd]_{f}\ar@{->}[rrrr]^{h}&&&& Y\ar@{->}[lld]^{g}\\
&& U &&}
$$
where $h$ is the contraction of all $f$-exceptional curves except $E_W$,
and $g$ is the contraction of the proper transform of $E_W$ on the surface $Y$.
Denote the $g$-exceptional curve by $E_Y$.
Then $Y$ has at most Du Val singularities of type~$\mathbb{A}$, the curve $E_Y$ is smooth,
and it contains all the singular points of the surface $Y$, if any.
One can check that the log pair~\mbox{$(Y,E_Y)$} has purely log terminal singularities
(see~\cite[Theorem~4.15(3)]{KoMo98}). Also, the divisor $-(K_Y+E_Y)$ is $g$-ample.
Thus, the curve $E_Y$ is a Koll\'ar component of the singularity $U\ni P$.
{Moreover, $$
\mathrm{Diff}_{E_Y}(0)=
\left\{\begin{array}{ll}
 0 &  \text{in the case of } \mathbb{A}_1,\\
\frac{m-1}{m}P_{m-1}& \text{in the case of } \mathbb{A}_m \text{ and }  k=1,\\
\frac{k-1}{k}P_{k-1}+\frac{m-k}{m-k+1}Q_{m-k}& \text{in the case of } \mathbb{A}_m \text{ and }  2\leqslant k\leqslant\frac{m+1}{2},\\
\frac{1}{2}P_1+\frac{1}{2}Q_1+\frac{m-3}{m-2}R_{m-3}& \text{in the case of } \mathbb{D}_m,\\
\frac{1}{2}P_1+\frac{2}{3}Q_2+\frac{m-4}{m-3}R_{m-4}& \text{in the case of } \mathbb{E}_m,  \\
              \end{array}
          \right.
$$
where $P_i$, $Q_j$, and $R_\ell$ are singular points of $Y$ that lie on $E_Y$. The singular point $P_i$ (resp. $Q_j$ and $R_\ell$)  is a Du Val singular point of type $\mathbb{A}_{i}$ (resp. $\mathbb{A}_{j}$ and $\mathbb{A}_{\ell}$). Since $E_Y\cong\mathbb{P}^1$, it follows from Example~\ref{example:P1-log} that
$$
\alpha\big(E_Y,\mathrm{Diff}_{E_Y}(0)\big)=
\left\{\begin{array}{ll}
 \frac{k}{m+1}\leqslant \frac{1}{2} &  \text{in the case of } \mathbb{A}_m,\\
1& \text{in the case of } \mathbb{D}_m, \\
2& \text{in the case of } \mathbb{E}_6, \\
3& \text{in the case of } \mathbb{E}_7, \\
6& \text{in the case of } \mathbb{E}_8. \\
              \end{array}
          \right.
$$}
\end{example}

\begin{example}
\label{example:A1}
{Let $U\ni P$ be a germ of a Du Val singularity of type $\mathbb{A}_m$,
and let $f\colon W\to U$ be the minimal resolution of this singularity.}

{Let $Q$ be a point on one of the $f$-exceptional
curves. We consider two cases, one is the case where the point $Q$ belongs to one of the two exceptional curves that correspond to ``tails'' of the dual graph  but it is not contained in any other exceptional curve, the other is the case where $Q$ is the intersection point of the $k$-th and $(k+1)$-th
$f$-exceptional curves, $1\leqslant k\leqslant \frac{m}{2}$.}

{Let $\xi\colon \widehat{W}\to W$ be the blow up at  $Q$,
and $\zeta$ be the contraction of the proper transforms of all  the $f$-exceptional
curves. Thus, there exists a commutative diagram
$$
\xymatrix{
\widehat{W}\ar@{->}[rrd]_{f\circ\xi}\ar@{->}[rrrr]^{\zeta}&&&& Y\ar@{->}[lld]^{g}\\
&& U&&}
$$
Denote the $g$-exceptional curve by $E_Y$.  It is a smooth rational curve. The dual graphs of the exceptional curves of the minimal resolution
of singularities $\zeta\colon \widehat{W}\to Y$ are chains such that the self-intersection numbers of the exceptional curves are $-3, -2, \ldots, -2$,
and the proper transform of $E_Y$ intersects only the ``tail'' components of these chains.
In the former case, $Y$ has a unique singular point $O$, which is a quotient of $\mathbb{C}^2$ by the cyclic group $\boldsymbol{\mu}_{2m+1}$. In the latter case, it contains two singular points $P_1$ and $P_2$, which are quotients of $\mathbb{C}^2$ by the cyclic groups $\boldsymbol{\mu}_{2k+1}$ and $\boldsymbol{\mu}_{2(m-k)+1}$, respectively.}

{By~\cite[Theorem~4.15(3)]{KoMo98} the log pair~\mbox{$(Y,E_Y)$} has purely log terminal singularities. Also, the divisor $-(K_Y+E_Y)$ is $g$-ample.
Thus, the curve $E_Y$ is a Koll\'ar component of the singularity $U\ni P$.
Moreover, $$
\mathrm{Diff}_{E_Y}(0)=
\left\{\begin{array}{ll}
\frac{2m}{2m+1}O& \text{in the former case},\\
\frac{2k}{2k+1}P_{1}+\frac{2(m-k)}{2(m-k)+1}P_2& \text{in the latter case}.\\
              \end{array}
          \right.
$$
Therefore,
$$
\alpha\big(E_Y,\mathrm{Diff}_{E_Y}(0)\big)=
\left\{\begin{array}{ll}
\frac{1}{2m+2}<\frac{1}{2}& \text{in the former case},\\
\frac{2k+1}{2m+2}\leqslant\frac{1}{2}& \text{in the latter case}.\\
              \end{array}
          \right.
$$
In particular, in the latter case we see that
$\alpha(E_Y,\mathrm{Diff}_{E_Y}(0))=\frac{1}{2}$ if and only if $m$ is even,
and $Q$ is the ``central point'' of the configuration of the $f$-exceptional curves.}
\end{example}

It is easy to see from \cite[Theorem~4.15]{KoMo98} that if $U\ni P$ is a Du Val singularity of type~$\mathbb{D}$
or~$\mathbb{E}$, and the exceptional curve $E_W$ in Example~\ref{example:Du-Val-different} is not the one corresponding to the ``fork'' of the dual graph,
then the curve $E_Y$ is not a Koll\'ar component  {(see \cite[Example~4.7]{Prokhorov})}.
We will see later that in these cases the singularity $U\ni P$ has a unique Koll\'ar component,
which is  described in Example~\ref{example:Du-Val-different}.
This is not true in general, i.e., a Koll\'ar component of a singularity $U\ni P$ may not be unique,
as one can see from Examples~\ref{example:Du-Val-different} and ~\ref{example:A1}.
Nevertheless, Li and Xu established in \cite[Theorem~B]{LiXu} the following:

\begin{theorem}\label{theorem:Li-Xu}
A K-semistable Koll\'ar component of $U\ni P$ is unique if it exists.
\end{theorem}

The K-semistable Koll\'ar components of two-dimensional Du Val singularities are described in our Examples~\ref{example:Du-Val-different} and~\ref{example:A1}.
They are precisely the Koll\'ar components whose $\alpha$-invariants are at least $\frac{1}{2}$ (cf. \cite[Example~4.7]{Liu}).

Note that Du Val singularities are two-dimensional rational quasi-homogeneous isolated hypersurface singularities.
The K-semistable Koll\'ar components of many three-dimensional rational quasi-homogeneous isolated hypersurface singularities
have been described in \cite{ChPaSh10,ChSh2013zoo}.
Similarly, the K-semistable Koll\'ar components of many four-dimensional rational quasi-homogeneous isolated hypersurface singularities
have been described in \cite{JohnsonKollar}.

The purpose of this paper is to prove the following analogue of Theorem~\ref{theorem:Park}.

\begin{theorem}\label{theorem:main} {Let $U\ni P$ be a germ of a Kawamata log terminal singularity.
Suppose that there is a commutative diagram
$$
\xymatrix{
X\ar@{->}[drr]_{\phi}\ar@{-->}[rrrr]^{\rho}&&&&\ Y\ar@{->}[dll]^{\psi}\\%
&&U&&}
$$
where $\phi :(X\supset E_X)\to ( U\ni P)$  and  $\psi :(Y\supset E_Y)\to ( U\ni P)$ are purely log terminal blow-ups of the germ $U\ni P$. If $$
\alpha\big(E_X,\mathrm{Diff}_{E_X}(0)\big)+\alpha\big(E_Y,\mathrm{Diff}_{E_Y}(0)\big)\geqslant 1,
$$
then $\rho$ is an isomorphism.}
\end{theorem}

Before proving this result, let us consider its applications.
Suppose that
\begin{equation}\label{eq:log-Fano-semi-stable}
\alpha\big(E_X,\mathrm{Diff}_{E_X}(0)\big)\geqslant\frac{\mathrm{dim}(U)-1}{\mathrm{dim}(U)}.
\end{equation}
By Theorem~\ref{theorem:main}, this inequality implies that the $\alpha$-invariant of another Koll\'ar component of the singularity $U\ni P$, if any,
must be less than $\frac{1}{\mathrm{dim}(U)}$, so that it should be K-unstable.
Of course, this also follows from Theorem~\ref{theorem:Li-Xu}, because the inequality~\eqref{eq:log-Fano-semi-stable}
implies that the log Fano variety $(E_X,\mathrm{Diff}_{E_X}(0))$ is K-semistable.

Theorem~\ref{theorem:main} also implies

\begin{corollary}
\label{corollary:Kollar-unique}
If $\alpha(E_X,\mathrm{Diff}_{E_X}(0))\geqslant 1$, then the Koll\'ar component of $U\ni P$ is unique.
\end{corollary}

This corollary is well known: it follows from \cite[Theorem~4.3]{Prokhorov} and \cite[Theorem~2.1]{Kudryavtsev}.
Recall from~\cite[Definition~4.1]{Prokhorov}
that the singularity $U\ni P$ is said to be \emph{weakly exceptional}
if it has a unique purely log terminal blow-up. This is equivalent to the condition that there is a Koll\'ar component $E_X$ of $U\ni P$
such that~\mbox{$\alpha(E_X,\mathrm{Diff}_{E_X}(0))\geqslant 1$} (see~\cite[Theorem~4.3]{Prokhorov}, \cite[Theorem~2.1]{Kudryavtsev}, and \cite{ChSh09}).
It follows from Example~\ref{example:Du-Val-different} that Du Val singularities of types~$\mathbb{D}$
and~$\mathbb{E}$ are weakly exceptional.
On the other hand, Du Val singularities of type  $\mathbb{A}$ are not weakly exceptional,
since each of them admits several Koll\'ar components (see Examples~\ref{example:Du-Val-different} and~\ref{example:A1}), and thus has several purely log terminal blow ups.

\begin{remark}
\label{remark:quotient}
Du Val singularities are special examples of two-dimensional quotient singularities.
Note that quotient singularities are always Kawamata log terminal.
For each of them, it is easy to describe one Koll\'ar component.
Let $\widehat{G}$ be a~finite subgroup in $\mathrm{GL}_{n+1}(\mathbb{C})$.
Suppose that $U\ni P$ is a~quotient singularity $\mathbb{C}^{n+1}\slash\widehat{G}$.
By the Chevalley--Shephard--Todd theorem, we may assume that the group $\widehat{G}$ does not contain any quasi-reflections (cf. \cite[Remark~1.16]{ChSh10}).
Let $\eta\colon\mathbb{C}^{n+1}\to U$ be the~quotient map.
Then there is a~commutative diagram
$$
\xymatrix{
W\ar@{->}[d]_{\pi}\ar@{->}[rr]^{\omega}&&Y\ar@{->}[d]^{\psi}\\
\mathbb{C}^{n+1}\ar@{->}[rr]_{\eta}&&U}
$$
where $\pi$ is the~blow up at the origin,
the~morphism $\omega$ is the~quotient map that is induced by the action~of $\widehat{G}$ lifted to the~variety $W$,
and $\psi$ is a~birational morphism.
Denote by $\widetilde{E}$ the exceptional divisor of $\pi$, and denote by $E_Y$ the exceptional divisor of $\psi$.
Then  $\widetilde{E}\cong\mathbb{P}^{n}$, and $E_Y$ is naturally isomorphic to the quotient $\mathbb{P}^{n}/G$,
where $G$ is the image of the group $\widehat{G}$ in $\mathrm{PGL}_{n+1}(\mathbb{C})$.
Moreover, the log pair $(Y,E_Y)$ has purely log terminal singularities, and the divisor $-(K_{Y}+E_Y)$ is $\psi$-ample.
Thus, the log Fano variety $(E_Y,\mathrm{Diff}_{E_Y}(0))$ is a Koll\'ar component of the singularity $U\ni P$.
Also, it follows from \cite[Example~7.1(1)]{LiXu} and \cite[Theorem~1.2]{LiXu} that {$(E_Y,\mathrm{Diff}_{E_Y}(0))$} is K-semistable.
Furthermore, one has
$$
\alpha\big(E_Y,\mathrm{Diff}_{E_Y}(0)\big)=\alpha_G\big(\mathbb{P}^n\big)
$$
(see~\cite[Proof of Theorem~3.16]{ChSh09}).
Thus, if $\alpha_G(\mathbb{P}^n)\geqslant 1$, then this Koll\'ar component is unique by Corollary~\ref{corollary:Kollar-unique}. One can find
many subgroups $G\subset\mathrm{PGL}_{n+1}(\mathbb{C})$ with $\alpha_G(\mathbb{P}^n)\geqslant 1$  in \cite{MarPr99,ChSh09,ChSh10,Sakovich2012,ChSh2013,Sakovich2014,ChSh2014crelle}.
Note also that one always has $\alpha_G(\mathbb{P}^n)\leqslant 1184036$ by \cite{Tiep}.
\end{remark}

In the remaining part of the paper, we prove Theorem~\ref{theorem:main}.
Let us use its assumptions and notations.
We have to show that $\rho$ is an isomorphism.
Suppose that this is not the case. Let us seek for a contradiction.

We may assume that $U$ is affine.
There exists a commutative diagram
$$
\xymatrix{
&&W\ar@{->}[rrd]^g\ar@{->}[lld]_f&&\\
X\ar@{->}[rrd]_{\phi}\ar@{-->}[rrrr]^{\rho}&&&& Y\ar@{->}[lld]^{\psi}\\
&& U &&}
$$
such that $W$ is a smooth variety, and $f$ and $g$ are birational morphisms.
Denote by  $E_X^W$ and $E_Y^W$ the proper transforms of the divisors $E_X$ and $E_Y$ on the variety $W$, respectively.
Then $E_X^W$ is $g$-exceptional, and $E_Y^W$ is $f$-exceptional.
We may assume that $E_X^W$, $E_Y^W$ and the remaining exceptional divisors of $f$ and $g$ form a divisor with simple normal crossings.

Observe that $E_X^W\ne E_Y^W$. Indeed, if $E_X^W=E_Y^W$, then $\rho$ is small,
which is impossible, because $-E_X$ is $\phi$-ample, and $-E_Y$ is $\psi$-ample (see \cite[Proposition~2.7]{Co95}).
Let $F_1,\ldots,F_m$ be the prime divisors on $W$ that are contracted by both $f$ and $g$.
Then
$$
K_W+E_X^W+aE_Y^W+\sum_{i=1}^{m}a_iF_i\sim_{\mathbb{Q}} f^*\Big(K_X+E_X\Big)
$$
for some rational numbers $a,a_1,\ldots,a_m$.
Since the log pair $(X,E_X)$ has purely log terminal singularities, all numbers $a,a_1,\ldots,a_m$ are strictly less than $1$.
Also, we have
$$
E_X^W\sim_{\mathbb{Q}} f^*\big(E_X\big)-bE_Y^W-\sum_{i=1}^{m}b_iF_i,
$$
where $b,b_1,\ldots,b_m$ are non-negative rational numbers.
Then $b>0$, since $f(E_Y^W)\subset E_X$.

Fix an integer $n\gg 0$.
Put $\mathcal{M}_X=|-nE_X|$.
Then $\mathcal{M}_X$ does not have any base points.
Denote its proper transforms on $Y$ and $W$ by $\mathcal{M}_X^Y$ and $\mathcal{M}_X^W$, respectively.
Then
$$
\mathcal{M}_X^W\sim_{\mathbb{Q}} -f^*\big(nE_X\big)\sim_{\mathbb{Q}}-nE_X^W-nbE_Y^W-\sum_{i=1}^{m}nb_iF_i,
$$
which implies that $\mathcal{M}_X^Y\sim_{\mathbb{Q}}-nbE_Y$.
On the other hand, we have
$-(K_Y+E_Y)\sim_{\mathbb{Q}} -\delta_Y E_Y$ for some positive  rational number $\delta_Y$. Put $\epsilon_X=\frac{\delta_Y}{nb}$.
Then $\epsilon_X\mathcal{M}_X^Y\sim_{\mathbb{Q}}-(K_Y+E_Y)$, so that
$$
K_W+E_Y^W+\epsilon_X\mathcal{M}_X^W+\alpha E_X^W+\sum_{i=1}^{m}\alpha_iF_i\sim_{\mathbb{Q}} g^*\Big(K_Y+E_Y+\epsilon_X\mathcal{M}_X^Y\Big)\sim_{\mathbb{Q}} 0
$$
for some rational numbers $\alpha,\alpha_1,\ldots,\alpha_m$.
Similarly, let $\mathcal{M}_Y$ be the base point free linear system $|-nE_Y|$.
Denote by $\mathcal{M}_Y^X$ and $\mathcal{M}_Y^W$ its proper transforms on $X$ and $W$, respectively.
Then there is a positive rational number $\epsilon_Y$ such that $\epsilon_Y\mathcal{M}_Y^X\sim_{\mathbb{Q}}-(K_X+E_X)$,
so that
$$
K_W+E_X^W+\epsilon_Y\mathcal{M}_Y^W+\beta E_Y^W+\sum_{i=1}^{m}\beta_iF_i\sim_{\mathbb{Q}} f^*\Big(K_X+E_X+\epsilon_Y\mathcal{M}_Y^X\Big)\sim_{\mathbb{Q}} 0
$$
for some rational numbers $\beta,\beta_1,\ldots,\beta_m$.

\begin{lemma}\label{lemma:alpha-beta-new}
One has $\alpha>1$ and $\beta>1$. In particular, the singularities of the log pairs $(Y,E_Y+\epsilon_X\mathcal{M}_X^Y)$ and $(X,E_X+\epsilon_Y\mathcal{M}_Y^X)$ are not log canonical.
\end{lemma}

\begin{proof}
It is enough to show that $\alpha>1$. We have
$$
E_Y^W+\epsilon_X\mathcal{M}_X^W+\alpha E_X^W+\sum_{i=1}^{m}\alpha_iF_i
\sim_{\mathbb{Q}} {-K_W }\sim_{\mathbb{Q}}
E_X^W+aE_Y^W+\sum_{i=1}^{m}a_iF_i-f^*\Big(K_X+E_X\Big).
$$
This gives
\begin{equation}\label{eq:epsilon-MX}
\epsilon_X\mathcal{M}_X^W
\sim_{\mathbb{Q}} (1-\alpha)E_X^W+(a-1)E_Y^W+\sum_{i=1}^{m}(a_i-\alpha_i)F_i-f^*\Big(K_X+E_X\Big).
\end{equation}
It implies that
$$
\epsilon_X\mathcal{M}_X\sim_{\mathbb{Q}}-(K_X+E_X)-(\alpha-1)E_X.
$$
Recall that $-(K_X+E_X)\sim_{\mathbb{Q}} -\delta_X E_X$. We then obtain
$$
\epsilon_X\mathcal{M}_X\sim_{\mathbb{Q}}-\Big(K_X+E_X\Big)-(\alpha-1)E_X\sim_{\mathbb{Q}}-t_X\Big(K_X+E_X\Big),
$$
where {$t_X=1+\frac{\alpha-1}{\delta_X}$}.
On the other hand, from~\eqref{eq:epsilon-MX} we obtain
$$
(1-\alpha)E_X^W+\sum_{i=1}^{m}(a_i-\alpha_i)F_i\sim_{\mathbb{Q}} (1-a)E_Y^W+ (1-t_X)f^*\Big(K_X+E_X\Big).
$$
Now we let
$$
B=(1-\alpha)E_X^W+\sum_{i=1}^{m}(a_i-\alpha_i)F_i+(a-1)E_Y^W,
$$
so that $-B$ is $f$-nef. Then $B$ is effective if and only if $f_{*}(B)=(1-\alpha)E_X$ is effective by Negativity Lemma (see \cite[Lemma~3.39]{KoMo98}).
Since $a<1$, the divisor $B$ is not effective, which implies that $\alpha>1$.
\end{proof}

As in the proof of Lemma~\ref{lemma:alpha-beta-new}, put {$t_Y=1+\frac{\beta-1}{\delta_Y}$}.
Then
$$
\epsilon_Y\mathcal{M}_Y\sim_{\mathbb{Q}}-t_Y\big(K_Y+E_Y\big).
$$
Now take any {positive} rational numbers $\lambda$ and $\mu$ such that $\lambda+\mu\geqslant 1$.
One has
$$
K_X+E_X+\lambda\epsilon_Y\mathcal{M}_Y^X+\mu\epsilon_X\mathcal{M}_X\sim_{\mathbb{Q}}-\big(\lambda+\mu t_X-1\big)\Big(K_X+E_X\Big),
$$
so that $K_X+E_X+\lambda\epsilon_Y\mathcal{M}_Y^X+\mu\epsilon_X\mathcal{M}_X$ is $\phi$-ample.
Similarly, we see that
$$
K_Y+E_Y+\lambda\epsilon_Y\mathcal{M}_Y+\mu\epsilon_X\mathcal{M}_X^Y\sim_{\mathbb{Q}}-\big(\lambda t_Y+\mu-1\big)\Big(K_Y+E_Y\Big),
$$
so that $K_Y+E_Y+\lambda\epsilon_Y\mathcal{M}_Y+\mu\epsilon_X\mathcal{M}_X^Y$ is $\psi$-ample.

\begin{lemma}
At least one of the log pairs $(X, E_X+\lambda\epsilon_Y\mathcal{M}_Y^X)$ and $(Y,E_Y+\mu\epsilon_X\mathcal{M}_X^Y)$ is not log canonical.
\end{lemma}

\begin{proof}
Suppose that both
$(X, E_X+\lambda\epsilon_Y\mathcal{M}_Y^X)$ and~\mbox{$(Y,E_Y+\mu\epsilon_X\mathcal{M}_X^Y)$} are log canonical.
Then the log pairs~\mbox{$(X, E_X+\lambda\epsilon_Y\mathcal{M}_Y^X+\mu\epsilon_X\mathcal{M}_X)$}
and~\mbox{$(Y,E_Y+\lambda\epsilon_Y\mathcal{M}_Y+\mu\epsilon_X\mathcal{M}_X^Y)$}
are also log canonical.
On the other hand, we have
$$
K_W+E_X^W+\lambda\epsilon_Y\mathcal{M}_Y^W+\mu\epsilon_X\mathcal{M}_X^W+cE_Y^W+\sum_{i=1}^{m}c_iF_i\sim_{\mathbb{Q}} f^*\Big(K_X+E_X+\lambda\epsilon_Y\mathcal{M}_Y^X+\mu\epsilon_X\mathcal{M}_X\Big)
$$
for some rational numbers $c,c_1,\ldots,c_m$ that do not exceed $1$. Similarly, we have
$$
K_W+E_Y^W+\lambda\epsilon_Y\mathcal{M}_Y^W+\mu\epsilon_X\mathcal{M}_X^W+dE_X^W+\sum_{i=1}^{m}d_iF_i\sim_{\mathbb{Q}} g^*\Big(K_Y+E_Y+\lambda\epsilon_Y\mathcal{M}_Y+\mu\epsilon_X\mathcal{M}_X^Y\Big),
$$
where $d,d_1,\ldots,d_m$ are rational numbers that do not exceed $1$.
Denote by $D_W$ the boundary $\lambda\epsilon_Y\mathcal{M}_Y^W+\mu\epsilon_X\mathcal{M}_X^W+E_X^W+E_Y^W+\sum_{i=1}^{m}F_i$.
Then
\begin{multline*}
K_W+D_W\sim_{\mathbb{Q}}f^*\Big(K_X+E_X+\lambda\epsilon_Y\mathcal{M}_Y^X+\mu\epsilon_X\mathcal{M}_X\Big)+(1-c)E_Y^W+\sum_{i=1}^{m}(1-c_i)F_i\sim_{\mathbb{Q}}\\
\sim_{\mathbb{Q}}g^*\Big(K_Y+E_Y+\lambda\epsilon_Y\mathcal{M}_Y+\mu\epsilon_X\mathcal{M}_X^Y\Big)+(1-d)E_X^W+\sum_{i=1}^{m}(1-d_i)F_i.
\end{multline*}
Moreover, the log pair $(W,D_W)$ is log canonical,
since $W$ is smooth, the linear systems
$\mathcal{M}_Y^W$ and $\mathcal{M}_X^W$ are free from base points,
and the divisors $E_X^W$, $E_Y^W$, $F_1,\ldots,F_m$ form a simple normal crossing divisor.
Since $K_X+E_X+\lambda\epsilon_Y\mathcal{M}_Y^X+\mu\epsilon_X\mathcal{M}_X$ is $\phi$-ample,
it follows from \cite[Corollary~3.53]{KoMo98} that
the log pair
$(X,E_X+\lambda\epsilon_Y\mathcal{M}_Y^X+\mu\epsilon_X\mathcal{M}_X)$
is the canonical model of the log pair $(W,D_W)$.
Similarly, the log pair
$(Y,E_Y+\lambda\epsilon_Y\mathcal{M}_Y+\mu\epsilon_X\mathcal{M}_X^Y)$
is also the canonical model of the log pair $(W,D_W)$,
because $K_Y+E_Y+\lambda\epsilon_Y\mathcal{M}_Y+\mu\epsilon_X\mathcal{M}_X^Y$ is $\psi$-ample.
Since the canonical model is unique by \cite[Theorem~3.52]{KoMo98},
we see that $\rho$ is an isomorphism.
Since  $\rho$ is not an isomorphism by assumption, we obtain a contradiction.
This completes the proof of the lemma.
\end{proof}

Let $\lambda=\alpha(E_X,\mathrm{Diff}_{E_X}(0))$ and $\mu=\alpha(E_Y,\mathrm{Diff}_{E_Y}(0))$.
We may assume that the log pair $(X, E_X+\lambda\epsilon_Y\mathcal{M}_Y^X)$ is not log canonical.
Then $(E_X, \mathrm{Diff}_{E_X}(0)+\lambda\epsilon_Y\mathcal{M}_Y^X\vert_{E_X})$
is not log canonical
by Inversion of adjunction, see \cite[17.6]{Asterisque}.
On the other hand, we have
$$
\epsilon_Y\mathcal{M}_Y^X\Big\vert_{E_X}\sim_{\mathbb{Q}} -\Big(K_X+E_X\Big)\Big\vert_{E_X}
\sim_{\mathbb{Q}} -\Big(K_{E_X}+\mathrm{Diff}_{E_X}(0)\Big).
$$
This is impossible by the definition of the $\alpha$-invariant $\alpha(E_X,\mathrm{Diff}_{E_X}(0))$.

\bigskip

\textbf{Acknowledgements.}
This work was initiated in New York in August 2017 when the authors attended the Conference on Birational Geometry at the Simons Foundation.
We would like to thank the Simons Foundation for its hospitality.
The paper was written while the first author was visiting the Max Planck Institute for Mathematics.
He would like to thank the institute for the excellent working condition.
The second author was supported by IBS-R003-D1, Institute for Basic Science in Korea.
The third author was supported by the Russian Academic Excellence Project ``5-100'',
by RFBR grants \mbox{15-01-02164} and 15-01-02158, by
the Program of the Presidium of the Russian Academy of Sciences No.~01 ``Fundamental Mathematics and
its Applications'' under grant \mbox{PRAS-18-01}, and by Young Russian Mathematics award.

We are grateful to the referees for helpful comments on the earlier version of the paper.

\end{document}